\documentclass{amsart}
\usepackage{graphicx}
\usepackage[mathscr]{eucal}
\usepackage{amssymb}
\newtheorem{thm}{Theorem}

\newtheorem{lem}[thm]{Lemma}

\theoremstyle{definition}

\theoremstyle{remark}

\newtheorem{ex}[thm]{Example}
\newcommand{\ML}{\mathcal{M}\mathcal{L}}
\newcommand{\F}{\mathcal{F}}
\newcommand{\Pm}{\mathrm{Prim}}
\newcommand{\mset}{\emptyset}
\newcommand{\Pme}{\mathrm{Prime}}
\newcommand{\MP}{Min-Primal}
\begin{document}
\baselineskip=18pt
\title{Topological properties of spaces of ideals of the minimal tensor product}
\author{Aldo J. Lazar}
\address{School of Mathematical Sciences\\
         Tel Aviv University\\
         Tel Aviv 69778, Israel}
\email{aldo@post.tau.ac.il}

\thanks{}%
\subjclass{46L06} \keywords{closed two sided ideal, prime ideal, minimal primal ideal, minimal tensor product}
\date{June 16, 2009}%
%\commby{}%
% ---------------------------------------------------------------
\begin{abstract}
   One shows that for two $C^*$-algebras $A_1$ and $A_2$ any continuous function on $\Pm(A_1)\times \Pm(A_2)$  can be
   continuously extended to $\Pm(A_1\otimes_{\mathrm{min}}~ A_2)$ provided it takes its values in a $T_1$ topological space.
   This generalizes \cite[Corollary 3.4]{B}. A new proof is
   given for a result of Archbold \cite{A} about the space of minimal primal ideals of $A_1\otimes_{\mathrm{min}} A_2$.
   To obtain these two results one makes use of the topological properties of the space of prime ideals of the
   tensor product.
\end{abstract}
\maketitle
% ------------------------------------------------------------
\section{Introduction and preliminaries} \label{S:intr}

The prime ideal space of $A_1\otimes A_2$, the minimal tensor product of two $C^*$-algebras $A_1$ and $A_2$, has
some interesting topological properties in relation with the prime ideal spaces of the factors: there is a
homeomorphism of $\Pme(A_1)\times \Pme(A_2)$ onto a dense subset of $\Pme(A_1\otimes A_2)$ and a continuous map of
the latter space onto the first which, with the obvious identification, is a retract onto $\Pme(A_1)~\times~
\Pme(A_2)$. It turns out that these maps can be useful in getting information on the structure of $A_1\otimes A_2$.
Usually one employs the primitive ideal space to this end but since we do not know if a retraction as above exists
in the case of $\Pm(A_1\otimes A_2)$, the primitive ideal space of $A_1\otimes A_2$, we have to use the prime ideal
space instead.

By identifying the commutant of $A_1\otimes \mathbf{1}$ in the multiplier algebra of $A_1\otimes A_2$ Brown showed
in \cite[Corollary 3.4]{B} that any bounded complex-valued continuous function on $\Pm(A_1)\times \Pm(A_2)$ has a
continuous extension to $\Pm(A_1\otimes A_2)$. The above mentioned retraction together with a device created by
Kirchberg in \cite{K} which completes a topological space with all its closed prime subsets allow us to find such an
extension for every continuous function whose range is a $T_1$ topological space.

Kaniuth proved in \cite{Ka} that if $A_1\otimes A_2$ has the property $(F)$ of Tomiyama then the minimal primal
space (see below for the definition) of $A_1\otimes A_2$ is canonically homeomorphic to \textup{\MP}$(A_1)\times$
\textup{\MP}$(A_2)$. Following that, Archbold proved in \cite{A} that the same conclusion is valid in a more general
situation than the presence of the property $(F)$. We give here a proof of this result of Archbold by using
topological methods.

For a topological space $X$ we denote by $\F(X)$ the collection of all its closed subsets. We endow $\F(X)$ with the
topology generated by all the families $\{F\in \F(X) \mid F\cap U\neq \mset\}$ where $U$ is an open subset of $X$.
If $X$ is a $T_0$ space then the map $x\to \overline{\{x\}}$ is a homeomorphism of $X$ into $\F(X)$. A subset $L$ of
$X$ is a limit set if there exists a net in $X$ that converges to all the points of $L$; by \cite[Lemme 9]{D} this
is the same as saying that each finite collection of open subsets that intersect $L$ has a non void intersection. By
Zorn's lemma every limit set is contained in a maximal (closed) limit set. The family of all maximal limit sets of
$X$ is denoted $\ML(X)$ and will be considered with its relative topology inherited from $\F(X)$. A non void closed
subset $F$ of $X$ is called prime if it is not the union of two closed subsets each different from $F$. Obviously,
for each $x\in X$, $\overline{\{x\}}$ is prime. A space is called point-complete if each closed prime subset of it
is the closure of a singleton. Following \cite{K} we shall denote by $X^c$ the family of all closed prime subsets of
a $T_0$ topological space $X$ endowed with the relative topology as a subfamily of $\F(X)$ and we shall call it the
point-complete envelope of $X$. It is indeed a point-complete $T_0$ space. The base space $X$ will be identified
with a subset of $X^c$.

Given a $C^*$-algebra $A$, an ideal of $A$ will always be a closed two sided ideal. We denote by $\mathrm{Id}(A)$
and $\mathrm{Id}^{\prime}(A) := \mathrm{Id}(A)\setminus \{A\}$. The topology of the space of primitive ideals,
$\Pm(A)$, is the usual hull-kernel topology and that of $\mathrm{Id}(A)$ is that one acquires by pulling the
topology of $\F(\Pm(A))$ when one associates to each closed subset of $\Pm(A)$ its kernel. The relative topology of
$\Pme(A)$ is also the hull-kernel topology and from here on by the hull of the ideal $I$, denoted $\mathrm{hull}I$,
we shall always mean the hull of $I$ in $\Pme(A)$. Clearly $\Pme(A)$ is $\Pm(A)^c$. An ideal $I$ of $A$ is called
primal, cf. \cite[Definition 3.1]{AB}, if for every finite family of $\mathrm{Id}(A)$ with at least two members and
zero product, $I$ contains one ideal of the family. An ideal is primal if and only if its hull is a closed limit
set, see \cite[Proposition 3.2]{A87}. There the hull is taken in the primitive ideal space but the same proof works
for prime ideals as well. Any primal ideal contains a minimal primal ideal (Zorn's lemma) and there is a one to one
correspondence between the family of all minimal primal ideals, \textup{Min-Primal}$(A)$, and $\ML(\Pme(A))$.

Let now $A_1$ and $A_2$ be $C^*$-algebras. For $I_j$ an ideal of $A_j$ we denote by $q_{I_j}$ the quotient map of
$A_j$ onto $A_j/I_j$. One defines the maps $\Phi,\Delta : \mathrm{Id}(A_1)\times \mathrm{Id}(A_2)\to
\mathrm{Id}(A_1\otimes A_2)$ by
\[
 \Phi(I_1,I_2) := ker(q_{I_1}\otimes q_{I_2}), \quad \Delta(I_1,I_2) := I_1\otimes A_2 + A_1\otimes I_2.
\]
Then $\Phi$ is a homeomorphism of $\mathrm{Id}^{\prime}(A_1)\times \mathrm{Id}^{\prime}(A_2)$ onto a dense subset of
$\mathrm{Id}^{\prime}(A_1\otimes A_2)$, see \cite[Theorem 6]{L}. Its restriction to $\Pme(A_1)\times \Pme(A_2)$ maps
it homeomorphically onto a dense subset of $\Pme(A_1\otimes A_2)$, see \cite[Lemma 2.13(v)]{BK} and \cite[Corollary
8]{L}. For $I$ an ideal of $A_1\otimes A_2$ one defines
\[
 I_{A_1} := \{a_1\in A_1 \mid a_1\otimes A_2\subset I\}, \quad I_{A_2} := \{a_2\in A_2 \mid A_1\otimes a_2\subset
 I\}
\]
and $\Psi(I) := (I_{A_1},I_{A_2})$. Then $\Psi : \mathrm{Id}(A_1\otimes A_2)\to \mathrm{Id}(A_1)\times
\mathrm{Id}(A_2)$ is continuous and $\Psi\circ \Phi$ restricted to $\mathrm{Id}^{\prime}(A_1)\times
\mathrm{Id}^{\prime}(A_2)$ is the identity map, see \cite[proof of Theorem 6]{L}. By this and \cite[Lemma 2.13]{BK}
$\Psi$ maps $\Pme(A_1\otimes A_2)$ onto $\Pme(A_1)\times \Pme(A_2)$.

\section{Extensions of continuous functions} \label{S:ext}

We begin with a simple lemma on extensions of continuous functions from a topological space to its point-complete
envelope.

\begin{lem} \label{L:env}

   Let $X$ be a $T_0$ topological space and $f$ a continuous function from $X$ into a $T_1$ space $Y$. Then $f$ has
   a (unique) continuous extension from $X^c$ to $Y$.

\end{lem}

\begin{proof}

   The function $f$ is constant on any prime closed subset of $X$. Indeed, if $S$ is such a subset and $f$ assumes two
   different values $y_1\neq y_2$ on $S$ then we choose open neighbourhoods $V_1$, $V_2$ of $y_1$, $y_2$
   respectively such that $y_1\notin V_2$ and $y_2\notin V_1$. Set now $S_1 := S\cap f^{-1}(Y\setminus V_1)$ and
   $S_2 := S\cap f^{-1}(Y\setminus V_2)$ and $\{S_1, S_2\}$ is a non-trivial decomposition of $S$.

   We define now $\tilde{f} : X^c\to Y$ by $\tilde{f}(S) := f(x)$ for $x\in S$. Then $\tilde{f}$ is well defined and
   it is an extension of $f$. For $U$ an open subset of $Y$ we have
   \[
    \{S\in X^c \mid \tilde{f}(S)\in U\} = \{S\in X^c \mid S\cap f^{-1}(U)\neq \mset\}
   \]
   and the continuity of $\tilde{f}$ is established.
\end{proof}

We come now to the generalization of \cite[Corollary 3.4]{B}. There the the functions were considered on the spectra
of the algebras; we prefer to work with the spaces of primitive ideals but, of course, there is no difficulty in
obtaining a version of the following result in terms of spectra.

\begin{thm}

   Let $A_1$ and $A_2$ be $C^*$-algebras and $\Phi : \mathrm{Id}^{\prime}(A_1)\times \mathrm{Id}^{\prime}(A_2)\to
   \mathrm{Id}^{\prime}(A_1\otimes~ A_2)$ be the canonical homeomorphism. Then for every $T_1$ topological space $Y$
   and any continuous function $f : (\Pm(A_1)\times \Pm(A_2))\to Y$, the function
   $f\circ \Phi^{-1} : \Phi(\Pm(A_1)\times \Pm(A_2))\to Y$ has a (unique) continuous
   extension from $\Pm(A_1\otimes~ A_2)$ to $Y$.

\end{thm}

\begin{proof}

   Lemma \ref{L:env} yields a continuous extension $\tilde{f} : (\Pm(A_1)\times \Pm(A_2))^c\to Y$ of $f$. By
   \cite[Proposition 7.9]{K} there is a homeomorphism $\nu$ from $\Pm(A_1)^c\times \Pm(A_2)^c = \Pme(A_1)\times
   \Pme(A_2)$ onto $(\Pm(A_1)\times \Pm(A_2))^c$ which is the identity on the copies of $\Pm(A_1)\times \Pm(A_2)$
   contained in these two spaces. Now, with $\Psi : \mathrm{Id}^{\prime}(A_1\otimes A_2)\to
   \mathrm{Id}^{\prime}(A_1)\times \mathrm{Id}^{\prime}(A_2)$ defined as in Section \ref{S:intr}, the function
   $\tilde{f}\circ \nu\circ \Psi : \Pme(A_1\otimes A_2)\to Y$ is continuous. The extension $\hat{f}$ which we need
   is the restriction of $\tilde{f}\circ \nu\circ \Psi$ to $\Pm(A_1\otimes A_2)$. Indeed, if $(P_1,P_2)\in
   \Pm(A_1)\times \Pm(A_2)$ then $\hat{f}(\Phi(P_1,P_2)) = f(\nu(P_1,P_2)) = f(P_1,P_2)$ since $\Psi(\Phi(P_1,P_2))
   = (P_1,P_2)$.

\end{proof}

\section{Minimal primal ideals} \label{S:min}

In this section we present our proof for Archbold's result \cite{A} on minimal primal ideals for tensor products.
The first step is a topological lemma.

\begin{lem} \label{L:retr}

   Let $X_1$, $X_2$, and $Y$ be topological spaces and $\phi$ a homeomorphism of $X_1\times X_2$ onto a dense subset
   $Z$ of $Y$. Suppose there is a continuous map $\psi : Y\to X_1\times X_2$ such that $\psi\circ \phi$ is the
   identity map of $X_1\times X_2$ and for each $(M_1,M_2)\in \ML(X_1)\times \ML(X_2)$, $\psi^{-1}(M_1\times M_2)$
   is the closure of $\phi(M_1\times M_2)$. Then $(M_1, M_2)\to \psi^{-1}(M_1\times M_2)$ is a homeomorphism, $\Theta$ say,
   of $\ML(X_1)\times \ML(X_2)$ onto $\ML(Y)$.

\end{lem}

\begin{proof}

   Obviously $\Theta$ is a one to one map.

   Let $(M_1,M_2)\in \ML(X_1)\times \ML(X_2)$. It is easily seen that $M_1\times M_2$ is a closed limit set of
   $X_1\times X_2$. Thus there exists a net in $Z$ that converges to all the points of
   $M := \overline{\phi(M_1\times M_2)} = \psi^{-1}(M_1\times M_2)$. Suppose now that $\{y\}\cup M$ is a limit set of $Y$.
   $Z$ is dense in $Y$ hence
   there exists a net $\{s_{\alpha}\}$ in $X_1\times X_2$ such that $\{\phi(s_{\alpha})\}$ converges to all the points of
   $\{y\}\cup M$. Then $\{s_{\alpha}\}$ converges to all the points of $\psi(y)\cup \psi(M) = \psi(y)\cup (M_1\times
   M_2)$. By using the canonical projections of $X_1\times X_2$ onto the factors we infer from the maximality of the
   limit sets $M_1$ and $M_2$ that $\psi(y)\in M_1\times M_2$ hence $y\in \psi^{-1}(M_1\times M_2)$. We have shown
   that the map $\Theta$ takes its values in $\ML(Y)$.

   Let now $L$ be a limit set in $Y$. As above, there is a net in $Z$ that converges to all the points of $L$ hence
   $\psi(L)$ is a limit set in $X_1\times X_2$. Another use of the canonical projections of the cartesian product
   shows that there exist maximal limit sets $M_1$, $M_2$ in $X_1$, $X_2$, respectively, such that $\psi(L)\subset
   M_1\times M_2$. Thus $L\subset \psi^{-1}(M_1\times M_2)$. We have shown that each maximal limit set of $Y$ is in
   the image of $\Theta$.

   If $U$ is an open subset of $Y$ then
   \begin{multline*}
      \{(M_1,M_2)\in \ML(X_1)\times \ML(X_2) \mid \Theta(M_1,M_2)\cap U\neq \mset\}\\ =
      \{(M_1,M_2)\in \ML(X_1)\times \ML(X_2) \mid \overline{\phi(M_1\times M_2)}\cap U\neq \mset\}\\
      = \{(M_1,M_2)\in \ML(X_1)\times \ML(X_2) \mid \phi(M_1\times M_2)\cap (U\cap Z)\neq \mset\}\\ =
       \{(M_1,M_2)\in \ML(X_1)\times \ML(X_2) \mid (M_1\times M_2)\cap \phi^{-1}(U\cap Z)\neq \mset\}.
   \end{multline*}
   There exist open sets $\{V^k_{\alpha}\}$, $k = 1,2$, such that $\phi^{-1}(U\cap Z) =
   \cup_{\alpha}(V^1_{\alpha}\times V^2_{\alpha})$. Thus
   \begin{multline*}
      \{(M_1,M_2)\in \ML(X_1)\times \ML(X_2) \mid \Theta(M_1,M_2)\cap U\neq \mset\}\\
      = \{(M_1,M_2)\in \ML(X_1)\times \ML(X_2) \mid (M_1\times M_2)\cap (\cup_{\alpha}(V^1_{\alpha}\times
      V^2_{\alpha}))\neq \mset\}\\
      = \{(M_1,M_2)\in \ML(X_1)\times \ML(X_2) \mid \cup_{\alpha}[(M_1\times M_2)\cap (V^1_{\alpha}\times
      V^2_{\alpha})]\neq \mset\}\\
      = \cup_{\alpha}[\{M_1\in \ML(X_1) \mid M_1\cap V^1_{\alpha}\neq \mset\}\times \{M_2\in \ML(X_2) \mid M_2\cap
      V^2_{\alpha}\neq \mset\}]
   \end{multline*}
   and the latter is an open set in $\ML(X_1)\times \ML(X_2)$. We conclude that $\Theta$ is continuous.

   Let now $V_k$ be open in $X_k$, $k = 1,2$; there exists an open set $W$ of $Y$ such that $\phi(V_1\times V_2) =
   Z\cap W$. We have
   \begin{multline*}
      \Theta(\{(M_1,M_2)\in \ML(X_1)\times \ML(X_2) \mid M_1\cap V_1\neq \mset, \quad M_2\cap V_2\neq \mset\})\\
      = \{\Theta(M_1,M_2)\in \ML(X_1\times X_2) \mid \phi(M_1\times M_2)\cap \phi(V_1\times V_2)\neq \mset\}\\
      = \{\Theta(M_1,M_2)\in \ML(X_1\times X_2) \mid \overline{\phi(M_1\times M_2)}\cap W\neq \mset\}
   \end{multline*}
   and this is an open subset of $\ML(X_1\times X_2)$. Thus we obtained that $\Theta$ is open and this concludes the
   proof.

\end{proof}

\begin{lem} \label{L:hulls}

   Let $A_1$ and $A_2$ be $C^*$-algebras and $I_1$, $I_2$ ideals in $A_1$, $A_2$, respectively. Then
   $\mathrm{hull}\Delta(I_1,I_2) = \Psi^{-1}(\mathrm{hull}I_1\times \mathrm{hull}I_2)$ and
   $\mathrm{hull}\Phi(I_1,I_2) = \overline{\Phi(\mathrm{hull}I_1\times \mathrm{hull}I_2)}$.

\end{lem}

\begin{proof}

   Suppose $P\in \mathrm{hull}\Delta(I_1,I_2)$; then $\Psi(P) = (P_{A_1},P_{A_2})\in \Pme(A_1)\times \Pme(A_2)$
   and $P_{A_1}\supseteq I_1$, $P_{A_2}\supseteq I_2$. Thus $P\in \Psi^{-1}(\mathrm{hull}I_1\times
   \mathrm{hull}I_2)$. Conversely, if $P\in \Pme(A_1\otimes A_2)$ and $\psi(P)\in \mathrm{hull}I_1\times \mathrm{hull}I_2$
   then $P\supseteq \Delta(P_{A_1},P_{A_2})\supseteq \Delta(I_1,I_2)$ and we got the reverse inclusion.

   The second equality is \cite[Corollary 3]{L}.

\end{proof}

The following result is an improvement obtained by Archbold of \cite[Theorem 1.1]{Ka}.

\begin{thm}[Theorem 4.1 of \cite{A}]  \label{T:min}

   Let $A_1$ and $A_2$ be $C^*$-algebras. If $\Phi(I_1,I_2) = \Delta (I_1,I_2)$ for all $(I_1,I_2)\in
   \emph{\MP}(A_1)\times \emph{\MP}(A_2)$ then $\Phi$ is a homeomorphism of
   $\emph{\MP}(A_1)\times \emph{\MP}(A_2)$ onto $\emph{\MP}(A_1\otimes A_2)$.

\end{thm}

\begin{proof}

   We shall exploit the fact that for a $C^*$-algebra $A$, the map $\mathrm{hull}(I)\to I$ is a homeomorphism of
   $\F(\Pme(A))$ onto $\mathrm{Id}(A)$ that maps $\ML(\Pme(A))$ onto \textup{\MP}$(A)$. Thus the conclusion
   will be obtained once we show that $(M_1,M_2)\to \mathrm{hull}\Phi(\mathrm{ker}M_1,\mathrm{ker}M_2)$ is a
   homeomorphism of $\ML(\Pme(A_1))\times \ML(\Pme(A_2))$ onto $\ML(\Pme(A_1\otimes A_2))$.

   By Lemma \ref{L:hulls}, the hypothesis on $A_1\otimes A_2$ is $\overline{\Phi(M_1\times M_2)} =
   \Psi^{-1}(M_1\times M_2)$. Thus the maps $\Phi : \Pme(A_1)\times \Pme(A_2)\to \Pme(A_1\otimes A_2)$ and $\Psi :
   \Pme(A_1\otimes A_2)\to \Pme(A_1)\times \Pme(A_2)$ satisfy the conditions of Lemma \ref{L:retr} which yields the
   desired homeomorphism.

\end{proof}

It is remarked in \cite[p. 142]{A} that there is no known example of $C^*$-algebras $A_1$, $A_2$ and minimal primal
ideals $I_1$, $I_2$ of these algebras such that $\Phi(I_1,I_2)\neq \Delta(I_1,I_2)$. By contrast, one constructs
easily an example of topological spaces $X_1$, $X_2$, and $Y$, a homeomorphism $\phi$ of $X_1\times X_2$ onto a
dense subset of $Y$, a continuous map $\psi : Y\to X_1\times X_2$ such that $\psi\circ \phi$ is the identity map of
$X_1\times X_2$ and maximal limit sets $M_1\subset X_1$, $M_2\subset X_2$ such that $\psi^{-1}(M_1\times M_2)\neq
\overline{\phi(M_1\times M_2)}$.

\begin{ex}

   Let $X_1 = X_2 := [0,1]$ with the usual topology and
   \[
    Y := ([0,1]\times [0,1])\cup \{y\}
   \]
   where $y$ is a point
   not in the square. A base for the topology of $Y$ consists of the topology of the square together with the family
   of all the sets $(U\setminus \{v\})\cup \{y\}$ where $v := (1,1)$ and $U$ runs through all the open
   neighbourhoods of $v$. Let $\phi$ be the identity map of the square and $\psi$ the map that is the identity on
   the square and takes $y$ to $v$. For $M_1 = M_2 := \{1\}$ we have $\overline{M_1\times M_2} = \{v\}$ but
   $\psi^{-1}(M_1\times M_2) = \{v,y\}$.

\end{ex}

\bibliographystyle{amsplain}
\bibliography{}

\end{document}